\documentclass[12pt]{article}    

\usepackage[margin=0.75in]{geometry} 
\usepackage{amsmath,amssymb,amsthm}

\newtheorem{theorem}{Theorem}[section]

\newtheorem{lemma}[theorem]{Lemma}

\theoremstyle{remark}

\newcommand{\s}{\sigma}

\newcommand{\be}{\begin{equation}}
\newcommand{\ee}{\end{equation}}

\newcommand{\bt}{\beta}

\newcommand{\vp}{{\mathsf{v}^{\prime}}}

\newcommand{\rc}{{\rm c}}
\newcommand{\bea}{\begin{eqnarray}}
\newcommand{\eea}{\end{eqnarray}}

\numberwithin{equation}{section}
\begin{document}

\title{Continuous and discrete Painleve equations arising from the gap probability distribution of the finite $n$ Gaussian Unitary Ensembles.}
\author{
Man Cao \\
Faculty of Science and Technology, Department of Mathematics,\\
University of Macau, Av. Padre Tom\'as Pereira, Taipa Macau, China\\
Yang Chen\;(yangbrookchen@yahoo.co.uk)\\
Faculty of Science and Technology, Department of Mathematics,\\
University of Macau, Av. Padre Tom\'as Pereira, Taipa Macau, China\\
James Griffin \; (jgriffin@aus.edu) \\
Department of Mathematics and Statistics \\
American University of Sharjah, Sharjah, UAE. PO Box 26666}


\date{\today}

\maketitle

\begin{abstract}
In this paper we study the gap probability problem in the Gaussian Unitary Ensembles of $n$ by $n$ matrices : 
The probability that the interval $J := (-a,a)$ is free of eigenvalues. In the works of Tracy and Widom, Adler and
 Van Moerbeke and Forrester and Witte on this subject, it has been shown that two Painleve type differential equations arise in this context. 
 The first is the Jimbo-Miwa-Okomoto $\sigma-$form and the second is a particular Painleve IV. Using the ladder operator technique of 
 orthogonal polynomials we derive three quantities associated with the gap probability, denoted by $\sigma_n(a)$, $R_n(a)$ and $r_n(a)$, 
 and show that each one satisfying a second order, non-linear, differential equation as well as a second order, non-linear difference equation. 
 In particular, in addition 
 to providing an elementary derivation of the aforementioned $\sigma-$form and Painleve IV we show that the quantity $r_n(a)$
  satisfies a particular case of Chazy's second degree second order differential equation. 
  For the discrete equations we show that the quantity $r_n(a)$ satisfies a particular form of the modified discrete Painleve II equation 
  obtained by Grammaticos and Ramani in the context of Backlund transformations. We also derive second order second degree difference equations
   for the quantities $R_n(a)$ and $\sigma_n(a)$.


\end{abstract}

\section{Introduction}

In their 1980 classic \cite{Jimbo-Miwa-Mori-Sato}, Jimbo, Miwa, Mori and Sato, studied the problem of 
impenetrable bosons in one dimension and showed that
in the thermodynamic limit, where the density is uniform, the probability of observing a gap (free of liquid) in the ground state may be
described by an allied quantity intimately related to the $\sigma$ function of a particular Painleve V ($P_{V}$). In the context of finite $n$ random
matrices, through their pioneering work, Tracy and Widom (1994) \cite{TracyWidom1994} , showed that the probability that the interval
$(-a,a)$ is absent of eigenvalues maybe characterized in terms of a particular $P_{IV}.$ In the context of gap probabilities, the study of finite $n$ ensembles of random matrices,
generated by Gaussian, Laguerre and Jacobi weights
for general values of $\bt-$ the so-called ``symmetry parameter", can be found in the work of Adler-Van Moerebeke \cite{Adler-Moer}.
\vskip 0.2cm
\noindent
Forrester and Witte \cite{ForresterWitte}, studied gap probabilities for finite n ensembles for the Cauchy weight and found second degree ODE's related to certain Painleve-VI transcendents.
\vskip 0.2cm
\noindent
Recently, Haine and Vanderstichelen, \cite{HaineVan},
in their study of the centerless representation of the Virasoro algebra associated with the circular unitary matrix ensembles, obtained a $P_{VI}$
without encountering a third order equation, which would require a first integral to reduce the order to two.
\vskip 0.2cm
\noindent
In this paper instead of setting up the gap probability as the determinant of $I$ minus
a Christoeffel-Darboux kernel augmented by
the characteristic function of $J=(-a,a)$, we study the Hankel determinant of order $n,$
generated by the (singularly) deformed weight
$$
w_0(x)\chi_{J^\rc}(x),\;{\rm where\;\;}  J^{\rc}=(-\infty,-a)\cup(a,\infty).
$$
Here $w_0(x)$ denotes the original or ``un-deformed" weight which in this case is taken to be ${\rm e}^{-x^2}.$
The standard theory, \cite{Mehta}, shows that the probability that $J$ is free of eigenvalues in the GUE is given by
\begin{equation} \label{dets}
{\rm Prob}(J^{\rc},n)=\frac{\det\left(\int_{\mathbb{R}}x^{i+j}{\rm e}^{-x^2}\chi_{J^{\rc}}(x)dx\right)_{0\leq i,j\leq n-1}}
{\det\left(\int_{\mathbb{R}}x^{i+j}{\rm e}^{-x^2}dx\right)_{0\leq i,j\leq n-1}}.
\end{equation}
We  write ${\rm Prob}(J^{\rc},n)$ as $\mathbb{P}(a,n).$ Again from standard theory, the determinant appearing in the denominator,
can be expressed as the product of the square of the $L^2( \mathbb{R},{\rm e}^{-x^2})$ norm of the {\it monic}
Hermite polynomials, denoted by  $h_n(0),\;\;n=0,1,2,..$. That is,
$$
\det\left(\int_{\mathbb{R}}x^{i+j}{\rm e}^{-x^2}dx\right)_{0\leq i,j\leq n-1}=
h_0(0)...h_{n-1}(0),
$$
where $$h_0(0)=\int_{\mathbb{R}}{\rm e}^{-x^2}dx$$ and
$$
h_k(0)=\frac{k!}{2^k}\:h_0(0),\;\;\;k=1,2,...
$$
\\
A similar approach has been applied to other singularly deformed weights for Unitary matrix ensembles where
$w_0$ is deformed to $w_0f,$ in \cite{ChenPruessner}, \cite{BasorChen2009}, \cite{BasorChenEhrhardt}, \cite{ChenZhang2010}, \cite{DaiZhang2010} and \cite{ChenMckay2010} for a variety of $f$.
\vskip 0.2cm
\noindent
Our results in this paper fall into two categories, that of continuous Painleve equations and that of discrete Painleve equations associated with the gap probability. In the realm of continuous Painleve equations we are able to provide an elementary derivation of the Sigma form that is satisfied by the logarithmic derivative of the probability. In our derivation the second order Sigma form comes out directly whereas in the Fredholm determinant approach in \cite{TracyWidom1994}, a third order equation results and a first integral of this equation is required in order to find the Sigma form. We are also able to obtain the $P_{IV}$ from \cite{TracyWidom1994} satisfied by a quantity allied to the gap probability. Finally for the continuous Painleve type equations we are able to show for the first time that a particular form of Chazy's second order second degree equation appears in this context.
\vskip 0.2cm
\noindent
A new feature of our work is the appearance of discrete equations in the context of gap probability. We derive second order difference equations for the Sigma quantity and also for the two allied quantities coming from the ladder operator approach. Two of these difference equations are second order second degree. The other equation is a second order first degree difference equation and is a particular case of the modified discrete Painleve II equation first discovered by Grammaticos and Ramani \cite{GrammaticosRamani} .
\vskip 0.2cm
\noindent
The paper is organized as follows. In section $2$ we apply the ladder operator approach to the Hermite weight on the set $J^{\rc}$ and define the auxiliary quantities related to the gap probability. In section $3$ we derive the second order non linear difference equations. In section $4$ we study the continuous evolution in $a$ and in section $5$ we present the derivation of the Sigma form, the $P_{IV}$ and the second order second degree equation of Chazy.

\section{Preliminaries}

Let $P_n(x)$ be monic polynomials of degree $n$ in $x$ and orthogonal,
 with respect to an even weight function  supported on $\mathbb{R},$ namely,
\begin{equation*}
\int_{-\infty}^{\infty}P_j(x)P_k(x)w(x)dx =h_j\delta_{jk}, \qquad j,k=0,1,2,...
\end{equation*}
Here,
\begin{equation*}
w(x):=w_0(x)\chi_{J^{\rc}}(x)=w_0(x)(1-\chi_{(-a,a)}),\;\;\;w_0(x):={\rm e}^{-\textsf{v}_0(x)},\;\;\textsf{v}_0(x)=x^2.
\end{equation*}
Our polynomials have the monomial expansion,
\begin{equation} \label{pn}
P_n(x)=x^n+\textsf{p}(n,a)x^{n-2}+..+P_n(0),
\end{equation}
and we shall later see that the coefficient of $x^{n-2}$ will play at important role.
\vskip 0.2cm
\noindent
From the orthogonality relation there follows the recurrence relation,
\bea
xP_n(x)=P_{n+1}(x)+\beta_n(a)P_{n-1}(x), \qquad n=1,2,...,
\eea
and we take the `initial' conditions to be  $\beta_0(a)P_{-1}(x):=0$, $P_1(x)=x$. We have indicated
the $a$ dependence of the recurrence coefficients $\bt_n$.
It is a simple consequence of the recurrence relation and orthogonality that
\begin{equation} \label{pb}
\textsf{p}(j,a)-\textsf{p}(j+1,a)=\bt_j(a).
\end{equation}
In terms of the parameter $a$, the numerator determinant from (\ref{dets}) reads,
$$
\det\left(\left(\int_{-\infty}^{-a}+\int_{a}^{\infty}\right)x^{i+j}{\rm e}^{-x^2}dx\right)_{0\leq i,j\leq n-1}
=h_0(a)h_1(a)...h_{n-1}(a).
$$
{\bf Remark} Although the polynomials also depend on $a$ we do not display this unless we have to.
\vskip 0.2cm
\noindent
\begin{lemma}
The monic orthogonal polynomials with respect to the weight $w$ on $\mathbb{R}$ satisfy the following structural relation
\begin{equation*}
P_n'(z)=\beta_n(a)\:A_n(z)\:P_{n-1}(z)-B_n(z)\:P_n(z),
\end{equation*}
where
\begin{equation}
\begin{split}
A_n(z)&=2+\frac{R_n(a)a}{z^2-a^2}\\
B_n(z)&=\frac{r_n(a)z}{z^2-a^2}
\end{split}
\end{equation}
and
\begin{equation}
R_n(a):=\frac{2w_0(a)P_n^2(a)}{h_n(a)},\;\;\;r_n(a):=\frac{2w_0(a)P_n(a)P_{n-1}(a)}{h_{n-1}(a)}.
\end{equation}
\end{lemma}
\begin{proof}
We start from
\begin{equation} \label{Pdiff}
P_n'(z)=\sum_{k=0}^{n-1}C_{n,k}P_k(z),
\end{equation}
and orthogonality implies,
\bea
C_{n,k}=\frac{1}{h_k}\int_{-\infty}^{\infty}P_n'(x)P_k(x)w(x)dx,\nonumber
\eea
Through integration by parts, and bearing in mind that
$w(x)=w_0(x)\chi_{J^{\rc}}(x),\;\;w_0(x)={\rm e}^{-x^2},$
it readily follows that
\begin{align}
C_{n,k}&=-\frac{1}{h_k}\int_{-\infty}^{\infty}P_n(x)P_k(x)w'(x)dx\nonumber\\
&=-\frac{1}{h_k}\int_{-\infty}^{\infty}P_n(x)P_k(x)({w_0'(x)\chi_{J^C}(x)+w_0(x)\chi_{J^C}'(x)})dx\nonumber\\
&=\frac{2}{h_k}\int_{-\infty}^{\infty}P_n(x)xP_k(x)w(x)dx+\frac{w_0(-a)}{h_k}P_n(-a)P_k(-a)-\frac{w_0(a)}{h_k}P_n(a)P_k(a).\nonumber
\end{align}
Note that $\beta_n(a)=\frac{h_n(a)}{h_{n-1}(a)}$. If we eliminate
eliminate $xP_k(x)$ from the above equation with the recurrence relations, we find,
\bea
C_{n,k}&=&\frac{w_0(a)}{h_k(a)}P_n(-a)P_k(-a)-\frac{w_0(a)}{h_k(a)}P_n(a)P_k(a), \qquad k=0,1,2,\ldots,n-2\nonumber\\
C_{n,n-1}&=&2\bt_n(a)+\frac{w_0(a)}{h_{n-1}(a)}P_n(-a)P_{n-1}(-a)-\frac{w_0(a)}{h_{n-1}(a)}P_n(a)P_{n-1}(a).\nonumber
\eea
The Christoffel-Darboux formula, \cite{Szego1939}, the discrete analogue of the Green's identity is given below.
\bea
\sum_{l=0}^{n-1}\frac{P_l(x)P_l(y)}{h_l}=\frac{P_n(x)P_{n-1}(y)-P_n(y)P_{n-1}(x)}{h_{n-1}(x-y)},
\eea
Applying this formula to (\ref{Pdiff}), with some computations we find
\begin{align}
P_n'(z)&=\sum_{k=0}^{n-1}C_{n,k}P_k(z)\nonumber\\
&=\sum_{k=0}^{n-1}\frac{w_0(a)}{h_k}P_n(-a)P_k(-a)P_k(z)-\sum_{k=0}^{n-1}\frac{w_0(a)}{h_k}P_n(a)P_k(a)P_k(z)+2\bt_nP_{n-1}(z)\nonumber\\
&=w_0(a)P_n(-a)\sum_{k=0}^{n-1}\frac{P_k(z)P_k(-a)}{h_k}-w_0(a)P_n(a)\sum_{k=0}^{n-1}\frac{P_k(z)P_k(a)}{h_k}+2\bt_nP_{n-1}(z)\nonumber\\
&=w_0(a)P_n(-a)\frac{P_n(z)P_{n-1}(-a)-P_n(-a)P_{n-1}(z)}{h_{n-1}(z+a)}+2\bt_nP_{n-1}(z)\nonumber\\
&\qquad \qquad \qquad \qquad \qquad \qquad -w_0(a)P_n(a)\frac{P_n(z)P_{n-1}(a)-P_n(a)P_{n-1}(z)}{h_{n-1}(z-a)}\nonumber\\
&=\bt_nP_{n-1}(z)\left[2-\frac{w_0(a)P_n^2(-a)}{h_{n}(z+a)}+\frac{w_0(a)P_n^2(a)}{h_{n}(z-a)}\right]\nonumber\\
&\qquad \qquad \qquad \qquad \qquad \qquad  -\left[\frac{w_0(a)P_n(a)P_{n-1}(a)}{h_{n-1}(a)(z-a)}-\frac{w_0(a)P_n(-a)P_{n-1}(-a)}{h_{n-1}(a)(z+a)} \right]P_n(z).\nonumber
\end{align}
Since the weight is even, $P_n(-z)=(-1)^nP_n(z),$ hence,
\bea
P_n'(z)=\bt_n(a)\left[2+\frac{2P_n^2(a)w_0(a)a}{h_n(a)(z^2-a^2)}\right]P_{n-1}(z) -\left[\frac{2w_0(a)P_n(a)P_{n-1}(a)z}{h_{n-1}(a)(z^2-a^2)}\right]P_n(z).
\eea
\end{proof}
\noindent
{\bf Remark}
$$
R_0(a)=\frac{{\rm e}^{-a^2}}{\int_{a}^{\infty}{\rm e}^{-x^2}dx},\;\;r_0(a)=0.
$$
\\
The well-known supplementary conditions $(S_1),\;$ $(S_2)$ and the ``Sum Rule" $(S_2')$ satisfied by the functions $A_n$ and $B_n$ continue to hold.
These are,
$$
B_{n+1}(z)+B_n(z)=zA_n(z)-\vp_0(z),\eqno(S_1)$$
$$
1+z(B_{n+1}(z)-B_n(z))=\bt_{n+1}A_{n+1}(z)-\bt_n\;A_{n-1}(z),\eqno(S_2)
$$
and
$$
\{B_n(z)\}^2+\vp_0(z)\:B_n(z)+\sum_{j=0}^{n-1}A_j(z)=\bt_n\:A_n(z)\:A_{n-1}(z).\eqno(S_2')
$$
The three identities  $(S_1),$ $(S_2)$  and $(S_2')$ valid for $z\in\mathbb{C}\bigcup\{\infty\}$ can be found in Magnus, \cite{Magnus1995}, and were fundamental in the work appearing in \cite{ChenIsmail1997}, \cite{ChenIsmail2005} and \cite{ChenIts2009}.
\\
By equating coefficients in $(S_1)$ and $(S_2')$ we obtain equations relating the recurrence coefficients $\beta_n(a)$ with the auxiliary quantities $R_n(a)$ and $r_n(a)$. From $(S_1)$, equating the residues at the simple pole produces
\begin{equation} \label{rnR}
r_{n+1}(a)+r_n(a)=a R_n(a).
\end{equation}
Replacing $n$ by $j$ in the above equation and summing from $j=0,$ to $j=n-1,$ gives,
\bea \label{sum-sum}
a\sum_{j=0}^{n-1}R_j(a)=2\:\sum_{j=0}^{n-1}r_j(a)+r_n(a),
\eea
which expresses $\sum_{j}r_j(a)$ in terms of $\sum_{j}R_j(a).$
\vskip 0.2cm
\noindent
From $(S_2')$, we obtain,
\begin{equation*}
\begin{split}
& \frac{z^2r_n^2(a)}{(z^2-a^2)^2}+\frac{2a^2r_n(a)}{z^2-a^2}+\sum_{j=0}^{n-1}\frac{a R_j(a)}{z^2-a^2}+2r_n(a)+2n \\
&\qquad \qquad =\frac{\bt_n a^2R_n(a)R_{n-1}(a)}{(z^2-a^2)^2}+\frac{2a\bt_n R_n(a)+2a\bt_n R_{n-1}(a)}{z^2-a^2}+4\bt_n.
\end{split}
\end{equation*}
Letting $z\to\infty,$ and equating the residues of the second order pole and of the first order pole, we obtain,
\begin{equation} \label{br}
\bt_n(a)=\frac{n}{2}+\frac{r_n(a)}{2}
\end{equation}
\begin{equation} \label{rbRn}
(r_n(a))^2=\bt_n(a)\:R_n(a)\:R_{n-1}(a)
\end{equation}
\begin{equation*}
\sum_{j=0}^{n-1}R_j(a)+2ar_n(a)=2\bt_n\left[R_n(a)+R_{n-1}(a)\right],
\end{equation*}
respectively.
\\
Summing $\textsf{p}(j,a)-\textsf{p}(j+1,a)=\bt_j(a)$ from (\ref{pb}) , from $j=0$ to $j=n-1,$ gives,
\begin{equation} \label{p1}
-\textsf{p}(n,a)=\sum_{j=0}^{n-1}\bt_j=\sum_{j=0}^{n-1}\frac{r_j(a)}{2}+\frac{n(n-1)}{4}.
\end{equation}
Note that,
$$
\sum_{j=0}^{n-1}r_j(a)=\frac{a}{2}\sum_{j=0}^{n-1}R_j(a)-\frac{r_n(a)}{2},
$$
which follows from (\ref{sum-sum}) and
$$
\sum_{j=0}^{n-1}R_j(a)=-2a\:r_n(a)+(n+r_n(a))R_n(a)+\frac{2\;(r_n(a))^2}{R_n(a)}.
$$
We obtain, after easy computations, an equation, which expresses $\textsf{p}(n,a)$ in terms of $r_n(a)$ and $R_n(a),$
\begin{equation} \label{prR}
-\textsf{p}(n,a)=\frac{n(n-1)}{4}-\left(\frac{1}{4}+\frac{a^2}{2}\right)r_n(a)+\frac{a}{4}\;(n+r_n(a))R_n(a)+
\frac{a}{2}\:\frac{(r_n(a))^2}{R_n(a)},
\end{equation}
which is crucial to later development.
\\
In summary we have the equations (\ref{rnR}), (\ref{br}), (\ref{rbRn}) and (\ref{prR}) which we will use in the next section to derive the second order difference equations for the quantities $r$ and $R$. We will also derive a second order difference equation for $\sigma$, a quantity that is related to the gap probability which we define in the next section.

\section{Non-linear difference equations satisfied by $r_n(a),\;\sigma_n(a),$ and $R_n(a)$.}

The modified discrete Painleve II equation of Grammaticos and Ramani, \cite{GrammaticosRamani}, has the following form
\begin{equation} \label{mdP2}
(x_{n-1}+x_n)(x_n+x_{n+1}) = \frac{-4x_n^2+m^2}{\lambda x_n + z_n}
\end{equation}
where $m, \lambda$ are constants and $z_n$ is a linear function of $n$.
Our first observation from the formulas appearing in the previous section is that a simple combination of the equations  (\ref{rnR}), (\ref{br}) and (\ref{rbRn}) leads immediately to the following second order difference equation for $r_n(a)$
\begin{equation} \label{rd2}
a^2\:(r_n(a))^2=\frac{(n+r_n(a))}{2}\:[r_{n+1}(a)+r_n(a)][r_n(a)+r_{n-1}(a)].
\end{equation}
We make the substitution
\begin{equation*}
r_n(a) = -\frac{a^2}{2}y_n
\end{equation*}
in (\ref{rd2}) to obtain
\begin{equation} \label{mdp2-r}
\frac{-4y_n^2}{y_n-\frac{2n}{a^2}} = (y_{n+1}+y_n)(y_n+y_{n-1}).
\end{equation}
Therefore we see that the expression
\begin{equation*}
-\frac{2r_n(a)}{a^2}
\end{equation*}
satisfies (\ref{mdP2}) with $m=0$, $\lambda=1$ and $z_n = -\frac{2}{a^2}n$. Note that (\ref{rd2}) may be iterated forward in $n$ starting from the `initial' conditions
$$
r_0(a)=0,\;\;\;r_1(a)=\frac{a{\rm e}^{-a^2}}{\int_{a}^{\infty}{\rm e}^{-x^2}dx}.
$$
Next we introduce the quantity $\sigma_n(a),$ defined by,
$$
\sigma_n(a):=-\sum_{j=0}^{n-1}R_j(a).
$$
We will see in the next section that $\sigma_n(a)$ is related to the gap probability by
\begin{equation}
\sigma_n(a) =\frac{d}{da}\ln\mathbb{P}(n,a).
\end{equation}
In terms of $\sigma_n(a)$, equation (\ref{p1}) is re-written as
$$
-\textsf{p}(n,a)=-\frac{1}{4}a\:\sigma_n(a)-\frac{r_n(a)}{4}+\frac{n(n-1)}{4}.
$$
Now because
$$
\beta_n(a)=\frac{n+r_n(a)}{2}=\textsf{p}(n,a)-\textsf{p}(n+1,a)
$$
we find,
$$
r_{n+1}(a)+r_n(a)=a\:(\sigma_n(a)-\sigma_{n+1}(a)),
$$
which gives the identification,
$$
R_n(a)=\sigma_n(a)-\sigma_{n+1}(a),
$$
providing a crucial link between $R_n(a)$ and $\sigma_n(a).$
\\
In terms of  $r_n(a)$ and $\sigma_n(a)$ (\ref{rd2}) becomes,
$$
\frac{2\:(r_n(a))^2}{r_n(a)+n}=R_n(a)\:R_{n-1}(a)=(\sigma_n(a)-\sigma_{n+1}(a))(\sigma_{n-1}(a)-\sigma_n(a)),
$$
a quadratic equation in $r_n(a),$ the solutions of which are expressed in terms of
$\sigma_n(a),\;\sigma_{n\pm1}(a),$ or equivalently, $R_n(a)R_{n-1}(a)$. Thus,
\begin{equation} \label{rRq}
r_n(a)=\frac{1}{4}\left(R_n(a)R_{n-1}(a)\pm\sqrt{R_n(a)R_{n-1}(a)}\:\sqrt{8n+R_n(a)R_{n-1}(a)}\right).
\end{equation}
Equation (\ref{prR}) may be re-written in terms of $\sigma_n(a),$ $r_n(a)$ and $R_n(a)$ as
\begin{equation} \label{srR2}
-\frac{a}{4}\sigma_n(a)-\frac{r_n(a)}{4}=-\left(\frac{1}{4}+\frac{a^2}{2}\right)r_n(a)+
\frac{a}{4}(n+r_n)R_n(a)+\frac{a}{2}\:\frac{(r_n(a))^2}{R_n(a)}.
\end{equation}
Eliminating $r_n$ in the above equation in favor of $R_n(a)\:R_{n-1}(a),$
we find, after computations a {\it second order} difference equation satisfied by $\sigma_n(a),$
where the gap variable $a$ appears as a parameter,
\begin{equation} \label{sd2}
\begin{split}
& (\sigma_n - \sigma_{n+1})(\sigma_{n-1}-\sigma_n)\left(2a+\sigma_{n+1}-\sigma_{n-1}\right)(2an+\sigma_n) = 2\left[\sigma_n+n(\sigma_{n-1}-\sigma_{n+1})\right]^2.
\end{split}
\end{equation}
Finally, we come to the second order difference equation satisfied by $R_n(a).$
Upon substituting (\ref{rRq}) into (\ref{rnR}), the desired difference equation for $R_n$ follows:
\begin{equation} \label{Rd2}
\begin{split}
& R_{n-1}R_{n+1}\left(R_nR_{n-1}+8n\right)\left(R_{n+1}R_n+8n+8\right) \\
& \qquad = \left[ 8R_na^2+R_nR_{n-1}R_{n+1}-4\left(aR_n+n+1\right)R_{n+1}-4\left(aR_n+n\right)R_{n-1}\right]^2
\end{split}
\end{equation}


\section{Evolution in $a.$}

In this section, we study the derivative of the quantities $h_n(a),\;$ and $\beta_n(a)$ with $a.$ We begin with
\bea
D_n(a)=h_0(a)\cdots h_{n-1}(a),\;\;\beta_n(a)=\frac{h_n(a)}{h_{n-1}(a)}.
\eea
Taking $\frac{d}{da}$ to
$$
h_n(a)=\int_{-\infty}^{\infty}\{P_n(x,a)\}^2w_0(x)\chi_{J^{\rc}}(x)dx,
$$
gives the following easily obtained expressions,
\bea
&&\frac{1}{h_n}\frac{d h_n(a)}{da}=-R_n(a),\nonumber\\
&&\frac{d}{da}\ln \bt_n(a)=\frac{1}{\bt_n}\frac{d\bt_n}{da}=R_{n-1}(a)-R_n(a),\nonumber\\
&&\frac{d}{da}\ln D_n(a)=\frac{d}{da}\log \mathbb{P}(n,a)=-\sum_{j=0}^{n-1}\:R_j(a).\nonumber
\eea
Note that the third quantity appearing above is $\sigma_n(a)$ from the previous section. Taking $\frac{d}{da}$ to
$$0=\int_{-\infty}^{\infty}P_{n}(x,a)P_{n-2}(x,a)w_0(x)\chi_{J^{\rc}}(x)dx$$
we find,
\begin{equation} \label{dpda}
\begin{split}
\frac{d\textsf{p}(n,a)}{da} &= \frac{2e^{-a^2}P_{n}(a)P_{n-2}(a)}{h_{n-2}(a)} \\
&= ar_n(a)-\bt_n R_n(a) \\
&= a\;r_n(a)-\frac{(n+r_n(a))}{2}\:R_n(a).
\end{split}
\end{equation}
The second last equality above is obtained by evaluating the recurrence relations at $x=a,$ and the
last equality follows from $\bt_n(a)=(n+r_n(a))/2.$
\\
Since $\textsf{p}(n,a)-\textsf{p}(n+1,a)=\bt_n$, we find,
\bea
\frac{d}{da}\bt_n(a)=a\left[r_n(a)-r_{n+1}(a)\right]-\bt_n R_n(a)+\bt_{n+1}R_{n+1}(a).\nonumber
\eea
Decreasing the index $n+1$ to $n$ with
\bea
&&r_{n+1}(a)=aR_n(a)-r_n(a)\nonumber\\
&&\beta_{n+1}R_{n+1}(a)=\frac{[aR_n(a)-r_n(a)]^2}{R_n(a)},\nonumber
\eea
followed by a straightforward computation gives us a Riccatti equation in $r_n(a)$.
\begin{equation} \label{drda}
\frac{d}{da}r_n(a)=\frac{2\:(r_n(a))^2}{R_n(a)}-(n+r_n(a))\:R_n(a).
\end{equation}
There is another Riccatti equation satisfied by $R_n(a)$ which maybe obtained as follows.
Taking the derivative with respect to $a$ on the equation (\ref{prR}) and adding equation (\ref{dpda}) gives 0. In this equation substitute for $r_n'(a)$ using (\ref{drda}). We find
$$
0=\{-2[\:r_n(a)]^2+(n+r_n(a))[R_n(a)]^2\}\{R'_n(a)-(R_n(a))^2+2a\:R_n(a)-4\:R_n(a)\}.
$$
This gives a further Riccatti equation, this time in $R_n(a)$
\bea
R_n'(a)=4\:r_n(a)+(R_n(a))^2-2a\:R_n(a).
\eea
\noindent
In summary, so far,we have the three important equations
\begin{equation} \label{eqnp}
-\textsf{p}(n,a)=\frac{n(n-1)}{4}-\left(\frac{1}{4}+\frac{a^2}{2}\right)r_n(a)+\frac{a}{4}\left(n+r_n(a)\right)R_n(a)+\frac{a}{2}\frac{r_n^2(a)}{R_n(a)}
\end{equation}
\begin{equation} \label{eqnrp}
\frac{dr_n(a)}{da}=\frac{2r_n^2(a)}{R_n(a)}-(r_n(a)+n) R_n(a)
\end{equation}
\begin{equation} \label{eqnpp}
\frac{d\textsf{p}(n,a)}{da}=ar_n(a)-\left[\frac{n}{2}+\frac{r_n(a)}{2}\right]R_n(a).
\end{equation}

\section{Riccatti, Chazy II and the $\sigma$ form}

To proceed further, we eliminate $r_n(a)$ from the Riccatti equation in $R_n(a)$, and find a
``dynamical" version of Painleve IV. The process is as follows. We begin with
\begin{equation} \label{Riccati}
R_n^{\prime}(a) - R_n(a)^2 +2a\:R_n(a)-4\:r_n(a) = 0.
\end{equation}
Differentiating with respect to $a$ gives
\begin{equation}
R_n^{\prime \prime}(a) - 2R_n(a)R_n'(a)+2R + 2a\:R_n^{\prime}(a)-4r_n^{\prime}(a) = 0.
\end{equation}
We now substitute for $r_n^{\prime}(a)$ using (\ref{eqnrp}) to find,
\begin{equation*}
R_n^{\prime \prime}(a) - 2R_n(a)R_n^{\prime}(a)+2R_n(a) + 2a\:R_n^{\prime}(a)-8\frac{r_n(a)^2}{R_n(a)}+4R_n(a)(n+r_n(a)) =0.
\end{equation*}
Using (\ref{Riccati}) we substitute for $r_{n}(a)$ in the above equation leading to a second order differential equation in $R_n(a)$ :
\begin{equation*}
\begin{split}
& R_n^{\prime \prime}(a) - 2R_n(a)R^{\prime}(a)+2R_n(a) + 2a\:R_n^{\prime}(a)-\frac{1}{2}
\frac{(R_n^{\prime}(a)-R_n^2(a)+2a\:R_n(a))^2}{R_n(a)} \\
& \qquad \qquad \qquad + 4n\:R_n(a)+R_n^{\prime}(a)R_n(a)-R_n^3(a)+2a\:R_n^2(a) =0.
\end{split}
\end{equation*}
This simplifies to
\begin{equation}
R_n^{\prime \prime} (a)= \frac{(R_n^{\prime}(a))^2}{2R_n(a)}+2(a^2-1-2n)R_n(a)-4a\:R_n^2(a)+\frac{3}{2}R_n^3(a).
\end{equation}
If we let
\begin{equation*}
y(a) = R_n(-a)
\end{equation*}
then
\begin{equation}
y^{\prime \prime}(a) = \frac{(y^{\prime}(a))^2}{2y(a)}+ \frac{\beta}{y(a)}+2(a^2-\alpha)y(a) +4a\:y^2(a)+\frac{3}{2}y^3(a).
\end{equation}
We see that $y(a)$ is a Painleve $IV$ with $\beta=0$ and $\alpha = 1+2n$.
\\
To obtain the Chazy equation, we proceed as follows: Solving for $R_n(a)$ in (\ref{eqnrp}), we find
$$
R_n(a)=\frac{-r_n'(a)\pm\sqrt{\Delta_n(a)}}{2(n+r_n(a))},
$$
where $\Delta_n(a):=(r'_n(a))^2+8\:r^2_n(a)(r_n(a)+a).$

\noindent
Choosing either sign and substituting the resulting $R_n$ into the Riccatti equation satisfied by $R_n(a),$ we find,
$$
 \frac{[-r_n'(a)+\sqrt{\Delta_n(a)}]\:[-8\:n\:r_n(a)-12\:r_n^2(a)+2\:a\:\sqrt{\Delta_n(a)}-r_n''(a)]}{[n+r_n(a)]\sqrt{\Delta_n(a)}}=0,
$$
resulting in an irrelevant algebraic equation and the following equation
$$
-8\:n\:r_n(a)-12\:r_n^2(a)+2\:a\:\sqrt{\Delta_n(a)}-r_n''(a)=0.
$$
Clearing the square root, followed by a translation in $n,$
$$
v_n(a):=-2r_n(a)-\frac{2n}{3},
$$
we have,
\bea
\left(v_n''(a)-6v_n^2(a)+\frac{8n^2}{3}\right)^2
=4\:a^2\left[(v_n'(a))^2-4v_n^3(a)+\frac{16n^2}{3}v_n(a)+\frac{64n^3}{27}\right],
\eea
the first member of the Chazy II system \cite{Cosgrove}.
\\
To obtain the Sigma form, we begin with equation (\ref{srR2})
$$
\frac{a}{4}\:\s_n(a)=\frac{a^2}{2}\:r_n(a)-\frac{a}{4}\:(n+r_n(a))R_n(a)-\frac{a}{2}\:\frac{(r_n(a))^2}{R_n(a)},
$$
which is simplified to
\begin{equation} \label{elim1}
(n+r_n(a))\:R_n(a)+\frac{2(r_n(a))^2}{R_n(a)}=2\:a\;r_n(a)-\s_n(a).
\end{equation}
We remind the Reader that
\bea
\s_n(a)=\frac{d}{da}\ln\mathbb{P}(n,a),
\eea
expressed in term of the derivative of the log of  the probability that the interval $(-a,a)$ is free of eigenvalues.
In terms of $\sigma_n(a)$, equation (\ref{eqnpp}) is written as
\bea
\frac{1}{4}(a\:\s_n(a))'=-\frac{r_n'(a)}{4} + ar_n(a) -\frac{1}{2}\:(n+r_n(a))\:R_n(a).
\eea
We eliminate $r_n'(a)$ with (\ref{eqnrp}) to give,
\begin{equation} \label{elim2}
\frac{1}{4}\:(a\s_n(a))'=a^2\:r_n(a) - \frac{a}{2}\:\frac{(r_n(a))^2}{R_n(a)}-\frac{a}{4}(n+r_n(a))R_n(a).
\end{equation}
In an easy computation, we eliminate $R_n(a)$ from equations (\ref{elim1}) and (\ref{elim2}) to find,
\bea
(a\s_n(a))'-\s_n(a)=2\:a\:r_n(a),
\eea
or
\bea
\s_n'(a)=2\;r_n(a).
\eea
We now have two equations, linear in $1/R_n(a)$ and $R_n(a);$ reproduced here
\bea
-(n+r_n(a))R_n(a)+\frac{2(r_n(a))^2}{R_n(a)}=r_n'(a)\nonumber\\
(n+r_n(a))R_n(a)+\frac{2(r_n(a))^2}{R_n(a)}=2\:a\:r_n(a)-\sigma_n(a).\nonumber
\eea
Solving for $1/R_n(a)$ and $R_n(a)$, gives,
\bea
\frac{4(r_n(a))^2}{R_n(a)}=2a\;r_n(a)-\sigma_n(a)+r_n'(a)\nonumber\\
2(n+r_n(a))R_n(a)=2\:a\:r_n(a)-\sigma_n(a)-r_n'(a),\nonumber
\eea
and the product of the above
\bea
8(n+r_n(a))(r_n(a))^2=(2a\;r_n(a)-\sigma_n(a))^2-(r_n'(a))^2.
\eea
Noting the fact that $r_n(a)=\sigma_n'(a)/2,$ we arrived at the Jimbo-Miwa-Okamoto $\sigma-$ form of a particular Painleve $IV;$
\bea
(\sigma_n''(a))^2=4(a\:\sigma_n'(a)-\s_n(a))^2-4(\s_n'(a))^2(\s_n'(a)+2n).
\eea

\end{document}